\documentclass[11 pt]{article}
\usepackage{amsmath, amsthm,amssymb,bbm,enumerate,mathrsfs,mathtools}
\usepackage{enumitem,a4wide}
\usepackage[margin=3cm]{geometry} 
\usepackage{tikz,verbatim,xcolor,graphicx}
\usepackage[linktocpage=true]{hyperref}
\usepackage{makecell}
\usepackage{verbatim,bbm,theoremref}

\renewcommand\labelenumi{(\roman{enumi})}
\renewcommand\theenumi\labelenumi
\newcommand{\eps}{\ensuremath{\varepsilon}}
\newcommand{\ind}{\mathbbm{1}}

\newcommand{\F}{\mathbbm{F}}

\newcommand{\E}{\mathbbm{E}}

\newcommand{\R}{\mathbbm{R}}

\newcommand{\Tr}{\mathrm{Tr}}
\renewcommand{\Pr}{\mathbbm{P}}

\newcommand{\Fqn}{\F_q^n}
\newcommand{\bx}{\mathbf{x} }

\newcommand{\bxi}[1]{\mathbf{x}^{(#1)}}

\newcommand{\sm}{\setminus}
\newcommand{\se}{\subseteq}

\newcommand{\xv}{\mathbf{x} }

\newcommand{\compl}[1]{{#1}^{\mathsf{c}}} 

\newcommand{\solA}[2]{\sol(#1;#2)}

\title{\vspace{-0.8cm}Towards a characterisation of  Sidorenko systems\thanks{Research supported by ARC Discovery Project DP180103684.}}
\author{Nina Kam\v{c}ev\thanks{Department of Mathematics, Faculty of Science, University of Zagreb, Croatia, {\tt nina.kamcev@math.hr}. Supported by European Union’s Horizon 2020 research and innovation programme (MSCA GA No 101038085).}
\and Anita Liebenau\thanks{School of Mathematics and Statistics, UNSW Sydney, NSW 2052, Australia, {\tt a.liebenau@unsw.edu.au}.}
\and Natasha Morrison\thanks{Mathematics and Statistics, University of Victoria, Victoria, B.C. Canada V8P 5C2, {\tt nmorrison@uvic.ca}. Supported by Natural Sciences and Engineering Research Council of Canada (NSERC Discovery Grant RGPIN-2021-02511).}
}

\newtheoremstyle{case}{}{}{\normalfont}{}{\itshape}{:}{ }{}
 
\newtheorem{thm}{Theorem}
\newtheorem{lem}[thm]{Lemma}

\newtheorem{conj}[thm]{Conjecture}

\newtheorem{claim}[thm]{Claim}
\newtheorem{ques}[thm]{Question}

\theoremstyle{definition}
\newtheorem{defn}[thm]{Definition}
\newtheorem{obs}[thm]{Observation}
\newtheorem{rem}[thm]{Remark}
\newtheorem{example}[thm]{Example}

\newtheoremstyle{case}{}{}{\normalfont}{}{\itshape}{\normalfont:}{ }{}

\theoremstyle{case}

\numberwithin{equation}{section}
\numberwithin{thm}{section}

\newcommand\given[1][]{\:#1  \vert  \:}

\newcommand{\sol}{\mathrm{sol}}

\begin{document}

\maketitle

\begin{abstract}
A system of linear forms $L=\{L_1,\ldots,L_m\}$ over $\F_q$ is said to be \emph{Sidorenko} if the number of solutions to $L=0$ in any $A \subseteq \F_{q}^n$ is asymptotically as $n\to\infty$ at least the expected number of solutions in a random set of the same density. Work of Saad and Wolf~\cite{sw17} and of Fox, Pham and Zhao~\cite{fpz19} fully characterises single equations with this property and both sets of authors ask about a characterisation of Sidorenko \emph{systems} of equations. 

In this paper, we make progress towards this goal. Firstly, we find a simple necessary condition for a system to be Sidorenko, thus providing a rich family of non-Sidorenko systems. In the opposite direction, we find a large family of structured Sidorenko systems, by utilizing the entropy method. We also make significant progress towards a full classification of systems of two equations. 
\end{abstract}

\section{Introduction}

A bipartite graph $H$ is called {\em Sidorenko} if the number of copies of $H$ in any graph $G$ is asymptotically at least the expected number of copies in a random graph with the same edge density as $G.$
Inspired by work resolving earlier conjectures of Erd\H{o}s~\cite{erdos62} from the 1960s and Burr and Rosta~\cite{br80} from the 1980s concerning the closely related property of commonness (see also~\cite{jst96,sid93,thomason89,thomason97}),  Sidorenko~\cite{sid93} conjectured that every bipartite graph is Sidorenko (an equivalent conjecture was earlier made by Erd\H{o}s and Simonovits~\cite{sim84}). It  has since become one of the major open problems in extremal combinatorics. Progress towards its resolution continues to this day. Sidorenko~\cite{sid93} proved the statement for complete bipartite graphs, trees and even cycles, but since then it has been verified for many other families of graphs (see~ \cite{sid1,sid2, sid3,sid4,sid5,sid6}).

The focus of this paper concerns the analogous questions for systems of equations. The study of these properties was initiated by Saad and Wolf~\cite{sw17}, whose motivation stemmed from existing earlier results for specific systems such as Schur triples and arithmetic progressions (see~\cite{datskovsky03,grr96,rz98,ss16,schoen99,w10}) as well as the extensive research in the graph setting. Let $q$ be a prime power and let $L=L(\xv)$ be a system of linear forms with coefficients in $\F_q.$ Say that $L$ is \emph{Sidorenko} if, for every subset $A \subseteq \Fqn$, we have that the number of solutions to $L(\bx)=0$ in $A$ is asymptotically, as $n\to\infty,$ at least the expected number of solutions in a random set of the same density. An equivalent, more practical definition is given below (see Definition~\ref{def:sid}). Say that $L$ is \emph{common} if, for every two-colouring of $\Fqn$, the number of monochromatic solutions is asymptotically, as $n\to\infty,$ at least the expected number of monochromatic solutions in a random colouring of $\Fqn.$  Clearly if a system is Sidorenko, then it is common.

When $L$ consists of a single linear form $a_1x_1 + \dots + a_kx_k$ with coefficients $a_i \in \F_q^\times = \F_q \setminus \{ 0\}$, these properties are well understood. Saad and Wolf~\cite{sw17} proved that such an $L$ is Sidorenko (and therefore common), whenever the coefficients can be partitioned into pairs, each summing to zero. They conjectured that this sufficient condition is also necessary, even for commonness, which was confirmed by Fox, Pham and Zhao~\cite{fpz19} using a novel construction via random Fourier coefficients. Furthermore, they showed that whenever $k$ is odd, $L$ is non-Sidorenko. An earlier cancellation argument due to Cameron, Cilleruelo and Serra~\cite{ccs07} shows that such $L$ is common whenever $k$ is odd. Hence, single homogeneous equations are fully characterised. Systems of multiple equations, however, appear to be more elusive.

Before stating our results, let us introduce some necessary notation. Let $L$ be a linear system consisting of $m$ linear forms $L_1,\ldots,L_m$ in $k$ variables with coefficients in $\F_q.$ For an $\ell$-variable system $L',$ we say that {\em $L$ induces $L'$ as a subsystem} if there exists a subset $\{i_1,\ldots,i_{\ell}\} \se[k]$ such that $L(x_1, \dots, x_k)=0$ implies that $L'(x_{i_1},\ldots, x_{i_{\ell}})=0.$ If $L'$ is a single linear form, we also say $L$ {\em induces the equation $L'=0$.} The \emph{length} of a linear equation $E$ is the number of variables in $E$ with \emph{non-zero} coefficients. Given a system $L$, let $s(L)$ denote the minimal length of an equation induced by $L$. 

A simple necessary condition for a graph $H$ to be Sidorenko is that it has no odd cycle, since a complete bipartite graph does not contain a copy of $H.$ Similarly, a simple necessary condition for a system to be Sidorenko is that the coefficients of each linear form sum to zero (we call such a system \emph{translation invariant}). Here, there is also a simple certificate, as the set $\{x \in \Fqn: x_1 = 1\}$ contains only solutions to translation-invariant systems. Our first result, proved in Section~\ref{sec:non-sid}, provides a non-trivial necessary condition for a system to be Sidorenko. Surprisingly, it tells us that if the shortest induced equation is of odd length, then the whole system is not Sidorenko. This generalises the fact that single equations with an odd number of variables are not Sidorenko~\cite{fpz19}. 

\begin{thm}\thlabel{thm:main1-odd}
Let $L$ be a system such that $s(L)$ is odd. Then $L$ is not Sidorenko. 
\end{thm}
As single-equation systems of odd length are common, we cannot hope that the condition that $s(L)$ is odd suffices for uncommonness. In~\cite{klm21-uncommon}, we prove the following related statement that concerns the case when $s(L)$ is even.

\begin{thm}[Theorem 1.3 in~\cite{klm21-uncommon}] \thlabel{t:uncommon}
Let $L$ be a system such that $s(L)$ is even. If all equations of length $s(L)$ induced by $L$ are uncommon, then $L$ is uncommon. 
\end{thm}
Since the property of being uncommon implies the non-Sidorenko property, this can be seen as an analogue of \thref{thm:main1-odd} for $s(L)$ even. Theorem~\ref{t:uncommon} follows from  Theorem 3.1 in~\cite{klm21-uncommon}, which provides a (rather general) sufficient condition for as system $L$ with $s(L)$ even to be uncommon.

In Section~\ref{sec:entropy}, we turn our attention to results in the positive direction. Say that a system $L=L(x_1,\ldots,x_k)$ defined by the linear forms $L_1,\ldots,L_m$  {\em admits a graph template} if there is a partition of the variables $x_1,\ldots,x_k$ into tuples $\bxi{1},\ldots,\bxi{r},$ for some $r$, such that every $L_{i}$ is of the form $L_{i}'( \bxi{u}) - L_{i}'( \bxi{v})$ for some $u\neq v.$ The graph $G$  on vertex set $[r]$ and with an edge $uv$ for every $L_i$ of the form $L_{i}'( \bxi{u}) - L_{i}'( \bxi{v})$ is called {\em a template graph for $L$.}

\begin{thm}\thlabel{thm:main2-entropy}
If $L$ is a system of linear forms that admits a graph template which is a tree, allowing parallel edges, $L$ is Sidorenko. 
\end{thm}

\begin{rem}
    It can easily be seen that if $L$ and $L'$ are systems that are both Sidorenko then their union on disjoint sets of variables is also Sidorenko, as solutions in subsets multiply. So in fact, `tree' can be replaced by `forest' in the theorem.
\end{rem}

To prove \thref{thm:main2-entropy} we adapt the entropy method to this setting. The entropy method was first used in the context of Sidorenko's conjecture for graphs in~\cite{sid6}, and has since been extensively applied in this setting, see~\cite{sid2,sid3,sid5}. We also recommend~\cite{gowers15web,tao17web} for excellent explanations of the underlying idea of the method. See also~\cite{ent-aps} for many other applications of the entropy method to combinatorics. 

In light of  \thref{thm:main1-odd,thm:main2-entropy} and our solid understanding of systems containing a single linear form, it is natural to wonder how close we are to a complete characterisation of Sidorenko systems in general. In Section~\ref{sec:2xk}, we develop  tools for analysing systems of rank two, and utilise them to make headway towards such a characterisation. Understanding such systems fully can be invaluable in trying to understand larger more complex systems, see Theorem 3.1 and Remark 3.2 in~\cite{klm21-uncommon}. \thref{cor:sidfun} provides sufficient conditions both to ensure that a two-equation system is common and to certify that it is uncommon. We apply this tool to give some examples of common systems with particular properties (see \thref{ex:non-AQ,ex:cont-AQ}) and, in \thref{p:sid-one}, to give a large family of non-Sidorenko two-equation systems that are not covered by \thref{thm:main1-odd}. These results illustrate some of the complexities involved in determining the properties of systems of multiple equations. In particular, we see that the multiplicative structure of the coefficients appears to play a key role. This is in stark contrast to the single equation case, where properties are determined solely by the additive structure of coefficients.

\subsection{Notation and definitions}\label{sec:pr}

Let $q$ be a prime power and let $L$ be a $k$-variable system defined by the linear forms $L_1,\ldots,L_m$ with coefficients in $\F_q.$ We identify $L$ with the $(m\times k)$-matrix  consisting of the coefficients of $L_1,\ldots, L_m.$ We call $L$ an $(m\times k)$-system if $m\le k$ and the coefficient matrix has full rank. Throughout, we work interchangeably with systems and their corresponding matrices. 

For a set $A \se \Fqn$ define the solution set of $L$ in $A$ to be 
$$\solA{L}{A} = \{\bx = (x_1, \dots, x_k) \in A^{k}: L \bx^T = 0 \},$$ 
and the {\em density of solutions} to $L$ in  $A$ to be 
\begin{equation}\label{eq:lamf}
\Lambda_{L} (A) = \frac{|\solA{L}{A}|}{|\solA{L}{\Fqn}|}.
\end{equation}
The following observation, though immediate, is useful throughout the paper.
\begin{obs}\label{ob:n-vec} 
	Let $L$ be an $(m\times k)$-system. 
	Then $|\solA{L}{\Fqn}| = |\solA{L}{\F_q}|^n = q^{n(k-m)}$. 
\end{obs}

Following the terminology from \cite{rr97}, we call systems that induce an equation $x_i=x_j,$ for some $i\neq j,$ {\em redundant systems} (and otherwise we call them {\em irredundant}). It is not difficult to verify (see~\cite{fpz19}) that the following definitions are equivalent to those given in the introduction for irredundant systems. 

\begin{defn}\thlabel{def:sid}
Let $L$ be an irredundant $(m \times k)$-system. Say that $L$ is \emph{Sidorenko} if for all $n$ and all $A \se \Fqn$, we have 
$$|\solA{L}{A}| \ge \frac{|A|^k}{q^{nm}},$$
or equivalently 
$$\Lambda_{L} (A) \ge \left(\frac{|A|}{q^n}\right)^k.$$
\end{defn}

\begin{rem}
Let $L(x_1,\ldots,x_k)$ be a system of linear forms and let $L'(x_1,\ldots,x_{k+1})$ be the system obtained from $L$ by including the form $x_k-x_{k+1}.$ It is easy to see that $L$ is Sidorenko if and only if $L'$ is. Thus it suffices to characterise irredundant systems. We would also like to point out that the benchmark for irredundant systems is no longer $\left(|A|/q^n\right)^k$, so \thref{def:sid} really only applies to the irredundant case.
\end{rem}

\section{Proof of Theorem~\ref{thm:main1-odd}}\label{sec:non-sid}

In this section we will show that any system $L$ with $s(L)$ odd is not Sidorenko. We first introduce some terminology that will help in the proof.

Recall that $s(L)$ denotes the minimal length of an equation induced by $L.$ Observe that if $L'$ is an equation induced by $L,$ where the subset $\{i_1,\ldots,i_{\ell}\}\se[k]$ corresponds to solutions to $L'=0,$ then the coefficient matrix of $L$ is equivalent to a matrix $M$ which, in some row,  is zero for every $j\in [k]\setminus \{i_1,\ldots,i_{\ell}\}.$ 
Given an $(m \times k)$-system $L$, say that a set $B \se [k]$ is $t$-\emph{rank-reducing} for $L$ if the matrix obtained from $L$ by deleting the columns indexed by $B$ has rank $m-t$. If a set is $t$-rank-reducing for some $t \ge 1$, say that it is \emph{rank-reducing}. Note that if $|B|<s(L)$, then $B$ cannot be rank-reducing. Call a set $B \se [k]$  \emph{good} if, for some $t \ge 0$, it is $(t+1)$-rank-reducing and $|B| = s(L) + t$.  
Let us make a couple of observations that follow from elementary linear algebra. Consider $B \subseteq [k]$ with $|B| = s(L) + t$ and $t \geq 0$. Then $B$ is $r$-rank-reducing, for some $0 \le r \le t+1$. The set $B$ is good if and only if there is a matrix equivalent to $L$ containing exactly $t+1$ rows supported on $B$ (i.e.~it is zero for all $j\in[k]\sm B$ in those $t+1$ rows). In particular, a set $B$ is good if any subset $B' \subseteq B$ such that $|B'| = s(L)-1$ has the property that the collection of variables $\{x_j: j \in B\setminus B'\}$ is uniquely determined by the variables $\{x_i: i \in B'\}$. This implies that if $s(L)-1$ variables in $B$ are 0, then so are all the variables in $B$. In addition, note that if a variable $x_j$ is determined by $\{x_i: i \in B\}$, then $B \cup \{ j\}$ is also good. 

We will utilise the following simple facts about good sets.

\begin{lem}\label{lem:good-sets} Let $B$ be a good set. 
\begin{enumerate}
    \item\label{it:subset-good} Every subset $B' \subseteq B$ with $|B'| \ge s(L)$ is good.
    \item\label{it:good-unique}There is a unique maximal good set containing $B$.
\end{enumerate}
\end{lem}
\begin{proof}
Suppose that $|B| = s(L) + t$ for some $t \ge 0,$ and let $B' = B \setminus \{i\}$ for some $i \in B$. As $B$ is good, it is $(t+1)$-rank-reducing and there is a matrix $L^*$ equivalent to $L$ containing exactly $t+1$ rows supported on $B$. Applying row operations to $L^*$, we can obtain $t$ rows supported on $B'$ and so $B'$ is good. This proves \ref{it:subset-good}.

Now let $B_1$, $B_2$ be two distinct maximal good sets. We claim that $|B_1 \cap B_2| \leq s(L)-2$, so no good set can be contained in both of them. To see that, assume the opposite and let $W\se B_1 \cap B_2$  be a set of size $s(L)-1.$ Let $j \in B_1 \setminus B_2$ and let $B' = W \cup \{j\}.$ Then $B'$ is a subset of $B_1$ of size $s(L)$, so $B'$ is good by \ref{it:subset-good}. By applying row operations to $L$, we may suppose that $L$ contains a row $R$ supported on $B'$. Removing the columns corresponding to $B_2$ decreases the rank of $L$ by $|B_2|-s(L)+1$, since $B_2$ is good, and leaves $R$ with a single non-zero co-ordinate in column $j$. Removing this column decreases the rank by one more. Thus, the set $B_2 \cup \{j\}$ is good, contradicting the maximality of $B_2$.
This proves \ref{it:good-unique}.
\end{proof}

We need one more ingredient for the proof of \thref{thm:main1-odd}. Let $L$ be an $(m \times k)$-system. For $n\ge 1,$ $B \subseteq [k]$ and $A \subseteq \Fqn$, let 
    $$t_B(L, A) = \frac{ | \{\bx \in \solA{L}{\Fqn} : x_i \in A \text{ for } i \in B\} | }{|\sol(L,  \Fqn)|}.$$ 
 In particular, $t_{\emptyset}(L, A) = 1$ and $t_{[k]}(L, A)= \Lambda_{L} (A)$ is the density of solutions in $A$. 
 
A straightforward application of the inclusion-exclusion principle gives
\begin{equation}\label{eq:in-ex}
    \Lambda_{L}(A) = \sum_{B \subseteq [k]} (-1)^{|B|}t_B(L,\compl{A}).
\end{equation}

\begin{proof}[Proof of \thref{thm:main1-odd}]
Let $L$ be a $k$-variable system defined by $m$ linear forms, for some integers $k,m.$ We may assume without loss of generality that the linear forms are linearly independent. Furthermore, we may assume that $m<k$ since otherwise $\solA{L}{\F_q}=\{ \bf{0}\}$ and the set $\F_q\sm \{0\}$ witnesses that $L$ is not Sidorenko. Thus, we identify $L$ with an $(m\times k)$-matrix of full rank which is $m.$ 

Let $n$ be sufficiently large. We will show that the set $A=\F_q^n\setminus \{0\}$ satisfies $$\Lambda_L(A) < (1 - q^{-n})^k =\left(\frac{|A|}{q^n}\right)^k.$$
For $B\se [k]$, let $t_B(0)= t_B(L,\{0\}),$ which is the density of solutions to $L=0$ where variables indexed by $B$ are zero. By Observation~\ref{ob:n-vec}, for any $t$-rank-reducing set $B \subseteq [k]$, we obtain $t_B(0) = q^{-n(|B| - t)}$.
    So in particular, if $|B| < s(L)$, $t_B(0) = q^{-n|B|}$. For a good set $B$ with $|B|=s(L)+b$, $t_B(0)  = q^{-n(s(L)+b-b-1)} = q^{-n(s(L)-1)}.$ Any set $B$ with $|B| > s(L)$ that is not good will give a $O(q^{-n\cdot s(L)})$ term.\footnote{Recall that if $s(L)-1$ of the variables in a good set $B$ are 0, then all the variables in $B$ must be 0; this is not the case if $B$ is not good. }
So from \eqref{eq:in-ex} we obtain,
    \begin{equation}\label{eq:sums}
    \Lambda_{L}(A) = \sum_{\substack{B\subseteq [k]\\ |B| < s(L)}}(-1)^{|B|}q^{-n|B|} + \sum_{\substack{B \text{ good }\\ |B| \ge s(L) }}(-1)^{|B|}q^{-n(s(L) -1)} + O(q^{-n\cdot s(L)}),
    \end{equation}
where the implicit constant in the big-O notation is independent of $n.$ 
We will show that for $n$ sufficiently large, the right hand side is less than $(1 - q^{-n})^k$.
We first note that $$(1 - q^{-n})^k = \sum_{\substack{B\subseteq [k]\\ |B| < s(L)}}(-1)^{|B|}q^{-n|B|} + O(q^{-n s(L)}).$$
Thus, it suffices to show that the second sum in~\eqref{eq:sums} is (sufficiently) negative. This will easily follow from Lemma~\ref{lem:good-sets}. As every good set is contained in a unique maximal good set $B$, and any $B' \subseteq B$ with $|B'| \ge s(L)$ is good, we have
    $$\sum_{\substack{B \text{ good }\\ |B| \ge s(L) }}(-1)^{|B|}= \sum_{\substack{B \text{ good }\\ \text{ maximal }}} \sum_{\substack{B' \subseteq B\\ |B'| \ge s(L)}}(-1)^{|B'|}.$$
    Now for a maximal good set $B$, as $s(L)$ is odd we have
    $$\Sigma(B) := \sum_{\substack{B' \subseteq B\\ |B'| \ge s(L)}}(-1)^{|B'|} = -\binom{|B|}{s(L)}+ \binom{|B|}{s(L) + 1} + \cdots +(-1)^{|B|}.$$
    If $s(L) \ge |B|/2$, then $\binom{|B|}{s(L) + 1} -\binom{|B|}{s(L)} <0 $ and so grouping the sum into pairs in this way, we see that $\Sigma(B)$ is negative. If $s(L)< |B|/2$, a similar argument shows that $\sum_{i=0}^{s(L) -1}(-1)^i \binom{|B|}{i}$ is positive, hence (by the binomial theorem) $\Sigma(B)$ is negative, and so the second sum in \eqref{eq:sums} equals $C q^{-n(s(L)-1)},$ where $C$ is negative, as required to complete the proof.
    \end{proof}

\section{Proof of \thref{thm:main2-entropy}}\label{sec:entropy}

The main tool used in the proof of \thref{thm:main2-entropy} is the entropy method. We begin the section with a few preliminaries. Let $X$ be a discrete random variable taking values in a finite set $S$. For $s \in S$, write $p_X(s)$ to denote $\Pr(X = s)$. The \emph{entropy} of $X$ is defined by 
$$H(X) = -\sum_{x \in S}p_X(s) \log(p_X(s)),$$
where here and throughout the rest of the section, the logarithm will be taken base 2. 
For a random variable $X$ and a function $f$ of $X$, we can construct a copy $Y$ of $X$ that is \emph{conditionally independent of $X$ given $f(X)$} as follows by defining the joint distribution by 
\begin{align*}
\Pr\big(\!(X,Y) = (x,y)\!\big)  
 =\ind_{f(x) = f(y)} \Pr\big(f(X)\!=\!f(x)\!\big)\,\Pr\big(\!X=x \!\given\! f(X)\! =\! f(x)\!\big)\,\Pr\big(\!X = y\!\given\! f(X)\! =\! f(x)\!\big).
\end{align*}
Notice that $Y$ has the same distribution as $X$ and satisfies the condition $f(X) = f(Y)$. 

We now recall some useful properties of entropy that will be used throughout the section. These can be found, for example, in~\cite{tao17web,tv2006}. 

\begin{enumerate}[label={(E\arabic*)}]
    \item \label{it:E1} $H(X) \le \log(|S|)$ and equality holds if and only if $X$ is uniformly distributed over $S$.
    \item If $Y$ is a function of $X$, then $H(X,Y) = H(X)$.
    \item If $Y$ is conditionally independent of $X$ given $f(X)$, then $H(X \given Y,f(X)) = H(X \given f(X))$. 
    \item \label{it:E4} Chain rule: $H(X,Y,Z) = H(Z) + H(Y \given Z) + H(X\given Y,Z)$.
\end{enumerate}

Let $X=(X_1,\ldots,X_k)$ be a random variable taking values in some finite product set $S=S_1\times\ldots\times S_k.$ For a subset $B=\{i_1,\ldots,i_{|B|}\}\se [k],$ we denote by $X_B$ the random variable  
$(X_{i_1},\ldots,X_{i_{|B|}})$ with marginal distribution  
$$\Pr\Big((X_{i_1},\ldots,X_{i_{|B|}}) = (x_{i_1},\ldots,x_{i_{|B|}})\Big) 
	= \sum_{ (x_i)_{i\not\in B}} \Pr\Big((X_{1},\ldots,X_{k}) = (x_1,\ldots,x_k)\Big).$$  

We recall that a linear form $L(x_1,\ldots,x_k) = a_1x_1 + \ldots + a_kx_k$ with $a_i \in \F_q \setminus \{0\}$ is Sidorenko if and only if the coefficients $a_i$ can be partitioned into pairs, each summing to zero, that is, if and only if 
$$L(x_1,\ldots,x_k) = L'(x_{i_1},\ldots,x_{i_{k/2}}) - L'(x_{j_1},\ldots,x_{j_{k/2}}),$$ 
where the two sets $\{i_1,\ldots,i_{k/2}\}$ and $\{j_1,\ldots,j_{k/2}\}$ partition $[k].$ The two subsets $\{{i_1},\ldots,{i_{k/2}}\}$ and $\{{j_1},\ldots,{j_{k/2}}\}$ turn out to play a special role for how we can glue two Sidorenko equations (and more generally systems) together to form a larger Sidorenko system. Here is the key technical definition we use. 

\begin{defn}[Blocks]\thlabel{def:block}
Let $L(x_1,\ldots,x_k)$ be an $(m\times k)$-system and let $B\se [k].$ We call $B$ {\em a block of $L$} if for every $n\ge 1$ and every $A\se\Fqn$ there is a random variable $X=(X_1,\ldots,X_k)$ taking values in $\solA{L}{A}$ such that 
\begin{enumerate}
\item\label{block1} $H(X)\ge \log(|A|^k/q^{nm}),$ and 
\item\label{block2} the marginal distribution of $X_B$ is uniform on $A^{|B|}.$
\end{enumerate}
\end{defn}
Let us note first that condition~\ref{block1} together with~\ref{it:E1} implies that $L$ is Sidorenko, so such a distribution $X$ may not exist for any $B\se [k].$ 

Let $L=L(x_1,\ldots,x_k)$ be an arbitrary $(m\times k)$-system over $\F_q,$ defined by the linear forms $L_i(x_1,\ldots,x_k)$ for $i\in[m].$ Then we denote by $(L,-L)$ the system consisting of forms $L_i(x_1,\ldots,x_k)-L_i(y_1,\ldots,y_k)$ for $i\in[m],$ where $\{x_1,\ldots,x_k\}$ and $\{y_1,\ldots,y_k\}$ are disjoint sets of variables. That is, $(L,-L)$ is a system of $m$ Sidorenko equations such that the cancelling pairs of each form match up. Such a system is Sidorenko; the proof for single-equation systems using the Cauchy-Schwarz Inequality in~\cite{sw17} directly generalises to systems of the form  $(L,-L),$ see~\cite{v21} for the details. We now prove a slightly stronger statement, that will be used as the base case for our inductive argument in \thref{thm:main2-entropy}. In addition, proving this as a standalone result will allow us to demonstrate the use of the entropy method here in a simple setting, before jumping in to all the details of the main proof.

\begin{thm}\thlabel{thm:ent2}
Let $L$ be an irredundant $(m \times k)$-system over $\F_q.$ Then the sets $[k]$ and $[2k]\sm [k]$ are blocks of the system $L^*=(L,-L).$ In particular, $L^*$ is Sidorenko. 
\end{thm}

\begin{proof}
Let $L_1,\ldots,L_m$ be $m$ linearly independent forms that define $L.$
Let $n\ge 1$ and $A \se \Fqn.$ For $t=(t_1,\ldots,t_m) \in (\Fqn)^m$, define 
$$S(t) = \{ u \in A^k: L_i(u) = t_i \text{ for each } i \in [m]\}.$$

To define the random variable $X=(X_1,\ldots,X_{2k})$ taking values in $\solA{L^*}{A},$ first pick $t\in(\Fqn)^m$ with probability $|S(t)|/|A|^k,$ then pick $u=(u_1,\ldots,u_k)$ and $v=(v_1,\ldots,v_k)$ uniformly at random in $S(t),$ both choices being independent. Note that this guarantees that $L(u)=L(v),$ and thus $(u,v)$ is indeed a solution to $L^*(u,v)=0.$ 
Writing $U$ for the random variable $X_{[k]}$ and $V$ for $X_{[2k]\sm[k]}$ we have, for fixed $u$ and $v$ in $A^k,$ 
\begin{align*}
    \Pr\big( (U,V) = (u,v) \big) 
    &=\ind_{\{L(u)=L(v)\}} \frac{|S(L(u))|}{|A|^k}\cdot \frac{1}{|S(L(u))|^2}\nonumber\\
    &=\ind_{\{L(u)=L(v)\}} \frac{1}{|A|^k |S(L(u))|}.
\end{align*}
Summing the right-hand-side over all $v\in A^k,$ we obtain that 
$$\Pr\big( U = u  \big) = \sum_{v\in S(L(u))} \Pr\big( (U,V) = (u,v) \big) 
= \frac{1}{|A|^k}$$
for every $u\in A^k,$ thus $U$ is uniform on $A^k.$ By symmetry, the same is true for $V.$ This verifies~\ref{block2} of \thref{def:block}. 
We now prove that~\ref{block1} holds as well. 

Note that by construction, $U$ and $V$ are conditionally independent given $L(U),$ and conditioned on $L(U)=t,$ $U$ and $V$ are identically distributed. Thus, using the properties of entropy \ref{it:E1}--\ref{it:E4}, we have
\begin{align}\label{eq:ent_eg}
H(U,V) &= H(U,V,L(U)) = H(U) + H(V \given U, L(U))
    = H(U) + H(V \given L(U))\nonumber \\
    &= H(U) + H(U \given L(U)) 
    = H(U) + H(U, L(U)) - H(L(U)) \nonumber \\
    &= 2H(U) - H(L(U)).
\end{align}
 
Now, $U$ is uniform on $A^k,$ so $H(U) =\log(|A|^k),$ by~\ref{it:E1}. By the same property, $H(L(U))\le \log(q^{nm}).$ We deduce that $H(X)=H(U,V)\ge \log(|A|^{2k}/q^{nm}),$ which is \thref{def:block}~\ref{block1}. 

Finally, since $X$ is a random variable taking values only in $\solA{L^*}{A}$ by construction, we get that $|\solA{L^*}{A}|\ge |A|^{2k}/q^{nm},$ again by~\ref{it:E1}. Thus, $L^*$ is Sidorenko.
\end{proof}

Next, we show that gluing a Sidorenko system (that has a block) and a symmetric system $L^*=(L,-L)$ in a certain way yields a new Sidorenko system.

\begin{thm}\thlabel{thm:entropy-tree}
Let $S=S(x_1,\ldots,x_k)$ be an $(m\times k)$-system and let $B=\{i_1,\ldots,i_{|B|}\}\se[k]$ be a block for~$S.$ Let $L$ be a $|B|$-variable system. Let $S^*$ be the system obtained from $S$ by including the equations $L(x_{i_1},\ldots,x_{i_{|B|}}) - L(x_{k+1},\ldots,x_{k+|B|})$. Then $S^*$ is Sidorenko. Moreover, every block of $S$ is a block of $S^*$ and $[k+|B|]\sm[k]$ is a block of $S^*.$
\end{thm}

\begin{proof}[Proof of \thref{thm:entropy-tree}]
Let $b=|B|$ and by reordering the variables we may assume without loss of generality that $B=[b].$ Suppose that $L$ is an $(m'\times b)$-system for some $m'\ge 1.$ Note that then $S^*$ is a $(k+b)$-variable system defined by $m+m'$ (linearly independent) equations. 

Let $n\ge 1$ and let $A\se\Fqn.$ Let $X=(X_1,\ldots,X_k)$ be a random variable taking values in $\solA{S}{A}$ such that $H(X)\ge \log(|A|^k/q^{nm})$ and the distribution of $X_B$ is uniform on $A^b,$ which exists since $B$ is a block for $S$, cf.\ \thref{def:block}. Now, let $Y$ be a  conditionally independent copy of $X_B$ given $L(X_B)=L(Y).$ In particular, $Y$ is independent of $X_{b+1},\ldots,X_k.$ We note that $L(X_B)$ is itself a random variable taking values in $(\Fqn)^{m'}.$ 

The variable $Y$ is uniformly distributed on $A^b$ as its distribution is identical to that of $X_B$. Furthermore, $(X,Y)$ takes values in $\solA{S^*}{A},$ since $X$ takes values in $\solA{S}{A}$ and $Y$ is constructed such that $L(X_B)=L(Y).$
 
As in the previous proof, properties (E1)-(E4) now imply that 
\begin{align}
H(X,Y) &= H(X,Y,L(X_B)) = H(X) + H(Y \given X, L(X_B))
    = H(X) + H(Y \given L(X_B))\nonumber \\
    &= H(X) + H(X_B \given L(X_B)) 
    = H(X) + H(X_B, L(X_B)) - H(L(X_B)) \nonumber \\
    &= H(X) +H(X_B) - H(L(X_B)) 
    \ge \log\left(\frac{|A|^k}{q^{nm}}\right)+\log\left(|A|^b\right)-\log\left(q^{nm'}\right),
\end{align}
by the assumption on $H(X)$, since $X_B$ is uniform on $A^b$ and by~\ref{it:E1} for both $H(X_B)$ and $H(L(X_B)).$
It follows that 
$$H(X,Y)\ge \log\left(\frac{|A|^{k+b}}{q^{n(m+m')}}\right),$$
so $|\solA{S^*}{A}|\ge |A|^{k+b}/q^{n(m+m')},$ again by~\ref{it:E1}. Since $n$ and $A$ are arbitrary, it follows that the $((m+m')\times (k+b))$-system $S^*$ is Sidorenko.

Finally, it is clear that if the marginal random variable $X_{B'}$ of $X$ is uniform on $A^{|B'|}$ for some $B'\se[k],$ then the same is true when viewing $X_{B'}$ as a marginal random variable of $(X,Y).$ Thus, any block of $S$ is a block of $S^*.$ 
\end{proof}

The proof of \thref{thm:main2-entropy} now follows easily by induction. \thref{thm:ent2} provides the base case of when the graph template consists of a single edge. In the inductive step, \thref{thm:entropy-tree} then allows us to add a leaf to a block of an existing Sidorenko system.

\begin{proof}[Proof of \thref{thm:main2-entropy}]

    We prove by induction on $r$ that if  $T$ is a tree on $[r]$ and $\xv^{(1)},\ldots, \xv^{(r)}$ is a partition of variables of $L$ into tuples so that every equation in $L$ is of the form 
    $L'(\xv^{(u)})=L'(\xv^{(v)})$ for some $uv \in E(T),$ then the subset $B_u\se [k]$ of indices corresponding to the tuple $\xv^{(u)}$ is a block for $L,$ for every $u\in [r].$  
    
    For $r=2,$ the associated graph template of $L$ is a set of parallel edges, each corresponding to one of the linear forms $L_1,\ldots,L_m,$ between two tuples. That is, each $L_i$ is of the form $L'(\xv^{(1)}) = L'(\xv^{(2)})$. \thref{thm:ent2} implies that both $\xv^{(1)}$ and $\xv^{(2)}$ are blocks for $L,$ and that $L$ is Sidorenko.
    
    For $r\ge 3,$ let $T$ be a template graph for $L$ which is a tree. Assume without loss of generality that the vertex $r$ is a leaf and that $r-1$ is its unique neighbour in $T.$ Let $L_1,\ldots,L_{m'}$ be the linear forms associated with this edge, i.e.~they are of the form $L_{j}'(\xv^{(r-1)})=L_{j}'(\xv^{(r)}).$ The remaining forms $L_{m'+1},\ldots,L_m$ of $L$ form a system, call it $L^{-}$ that admits $T'$ as a template graph, where $T'$ is obtained from $T$ by removing $r$ and all (parallel) edges between $r$ and $r-1.$ By induction, each $B_u$ is a block for $L^-$, for every $u\in[r-1].$ \thref{thm:entropy-tree} now implies that every $B_u$ is a block for $L,$ for every $u\in [r],$ and that $L$ is Sidorenko.
    This completes the proof. 
\end{proof}

\section{Properties of $(2 \times k)$-systems} \label{sec:2xk}

In this section, we are specifically concerned with systems defined by two (non-equivalent) linear forms, i.e.~with $(2\times k)$-systems. Of particular interest are those systems with $s(L)=k-1.$ The motivation behind studying these systems stems from 
Theorem~3.1 in~\cite{klm21-uncommon}. There, the uncommonness of a system $L$ is deduced from understanding subsystems $L_B$ which are  single-equation systems, which are fully understood, or $(2\times \ell)$-systems with $s(L_B)=s(L)=\ell-1.$ 
\thref{thm:main1-odd} implies that any Sidorenko system must have even $s(L).$ When $s(L)=2$, then $L$ is either redundant or not Sidorenko, by  Corollary~3.2 in~\cite{klm21-uncommon}. So here we mainly consider the case where $s(L)\ge 4$, and such a system is irredundant. We first present a general criterion and then apply it to the special case of $k=5.$ 

We will see that there are $(2\times 5)$-systems with $s(L)=4$ that are  common but not Sidorenko. Let us first give a more convenient definition of commonness, analogously to \thref{def:sid}.
\begin{defn}
Let $L$ be an irredundant $(m \times k)$-system. Say that $L$ is \emph{common} if for all $n$ and all $A \se \Fqn$, we have 
$$\Lambda_{L} (A)+\Lambda_{L} (\compl{A}) \ge 2^{1-k}.$$
\end{defn}

It will be convenient in this section to use the functional perspective on Sidorenko and common systems. For a system $L=L(x_1,\ldots,x_k)$ and a function $f:\Fqn\to \R$ define 
\begin{equation*}
\Lambda_{L} (f) := \frac{1}{|\solA{L}{\Fqn}|}\sum_{\bx \in \solA{L}{\Fqn}} f(x_1)f(x_2) \dots f(x_k), 
\end{equation*}
and note that this definition agrees with~\eqref{eq:lamf} when $f$ is the indicator function of a set $A.$  
For a function $f:\Fqn \rightarrow [0,1]$, and a linear form $E = a_1 x_1 + \ldots + a_kx_k$, let 
$$\tau_E(f):= \sum_{r \in \widehat{\Fqn}\setminus \{0\}}\prod_{i = 1}^{k} \widehat{f}(a_i r),$$
where $\widehat{\Fqn}$ is the group of characters of $\Fqn$ and $\widehat{f}(r) = \E_{x \in \Fqn}f(x)\exp(2\pi i/p)^{- \Tr(r\cdot x)}$ is the Fourier transform of $f$ (where $p$ is the characteristic of the field $\F_q$ and $\Tr:\F_q\to\F_p$ is the standard trace map). We note that $\widehat{f}(-r) =\overline{\widehat{f}(r)}$ always holds. 
The `twisted-convolution' Fourier identity  
\begin{equation}\label{eq:twisted}
    \Lambda_E(f) = (\E f)^k + \tau_E(f)
\end{equation}
easily follows from the definition, see also equation~(2) in~\cite{fpz19}. 

Let $k \ge 5$ and let $L$ be a $(2 \times k)$-system with $s(L) = k-1$, where $E_1$ and $E_2$ are two linear forms defining $L.$ Note that the condition $s(L)= k-1$ implies that, for every $i\in [k],$ some linear combination of $E_1$ and $E_2$ yields an equation $L_i$ induced by $L$ that has support $[k]\setminus \{i\}.$ Note also that $L_i$ is unique up to scaling, and scaling does not change the solution set $\solA{L_i}{A}$. 

So given a $(2 \times k)$-system $L$ with $s(L) = k-1$, we denote by $L_i$ the unique (up to scaling) linear form induced by $L$ with support  $[k]\setminus \{i\}.$ Clearly, any two of the $L_i$'s also define $L.$ The following says that we can deduce whether a system is common or uncommon by understanding these equations $L_i.$ 

\begin{lem}\thlabel{cor:sidfun}
    Let $k \ge 5$ be odd and let $L$ be a $(2 \times k)$-system with $s(L) = k-1$. 
    \begin{enumerate}
        \item If for all $n$ and $f:\Fqn \rightarrow [0,1]$, we have $\sum_{i=1}^{k} \tau_{L_i}(f) \ge 0$, then $L$ is common.
        \item If there exists an integer $n\ge 1$ and a function $f: \Fqn \rightarrow [0,1]$ with $\sum_{i=1}^{k} \tau_{L_i}(f) < 0$, then $L$ is uncommon. 
    \end{enumerate}
\end{lem}

This is a result in the spirit of Theorem 3.1 from \cite{klm21-uncommon}, it allows us to relate the problem of understanding the system $L$ to that of understanding the equations $L_i$. We demonstrate below how (ii) can be utilised and find a large class of $(2\times 5)$-systems with $s(L)=4$ that are uncommon, and thus not Sidorenko. We also use (i) to show commonness of two peculiar $(2\times 5)$-systems with $s(L)=4$ that demonstrate that finding a characterisation of common and Sidorenko systems of even just two equations is much more complicated.

\begin{proof}[Proof of \thref{cor:sidfun}]

Let $n\ge 1$ be an integer, let $A \subseteq \Fqn$ and let $\alpha = |A|/q^n$. We first claim that 
\begin{align}\label{eq:aux346}
\Lambda_L(A) + \Lambda_L(\compl{A}) = \alpha^k + (1- \alpha)^k + \sum_{i=1}^{k} \Lambda_{L_i}(A) - k\alpha^{k-1}.
\end{align}		
Indeed, using \eqref{eq:in-ex}, the fact that $k$ is odd and that $t_{[k]}(L,A) = \Lambda_L(A)$ we obtain that 
$$\Lambda_L(A) + \Lambda_L(\compl{A}) = \sum_{\substack{B \subseteq [k]\\ |B| \le k-1}}(-1)^{|B|}t_B(L,A).$$
For $|B| \le k-2 = s(L)-1$, Observation~\ref{ob:n-vec} gives $t_B(L,A) = |A|^{|B|}q^{-n|B|} = \alpha^{|B|}$. Then~\eqref{eq:aux346} follows by noticing that $t_{[k]\setminus \{i\}}(L,A) = \Lambda_{L_i}(A)$.

To see that (i) holds, let $f$ be the indicator function of $A$ and note that each $L_i$ is an equation in $k-1$ variables. Then~\eqref{eq:aux346} together with the hypothesis of (i) and~\eqref{eq:twisted} imply that $\Lambda_L(A) + \Lambda_L(\compl{A}) \ge \alpha^k + (1- \alpha)^k\ge 2^{k-1},$ by convexity. Since $n$ and $A$ were arbitrary this implies that $L$ is common. 

The argument for (ii) is essentially the same as showing that a functional definition of Sidorenko systems is equivalent to \thref{def:sid} (which is included in~\cite{fpz19}). We repeat it here for completeness. Let $f:\Fqn\to[0,1]$ be a function (for some $n\ge 1$) such that $\sum_i \tau_{L_i}(f)<0$. We may assume that  $\E f=1/2$, since $f$ can be replaced by the function $\frac 12 (1+ f - \E f)$, using the identity $\tau_L \left( \frac 12 (1+ f - \E f)\right) = 2^{-k} \tau_L(f)$.
Let $n'\ge 1$ be large enough and let $A\se \F_q^{n+n'}$ be a set chosen at random by including an element $(x,y)\in \Fqn \times \F_q^{n'}$ with probability $f(x).$  
Then the expected size of $A$ is $q^n/2.$ Furthermore, the expectation of $\sum_{i=1}^k\Lambda_{L_i}(A)$ is equal to $\sum_{i=1}^k\Lambda_{L_i}(f)+o_{n'}(1),$ where the $o_{n'}(1)$ term accounts for the proportion of solutions $\xv$ to $L(\xv)=0$ in $\Fqn$ with not all coordinates distinct, which goes to zero as $n'\to\infty$ (we use here that $s(L) >2$, which implies that $L$ is irredundant). 
Now, by \eqref{eq:twisted},
$$\sum_{i=1}^k\Lambda_{L_i}(f) = \sum_{i=1}^k\tau_{L_i}(f) + k (\E f)^{k-1} < k 2^{1-k},$$ by assumption. 
This together with~\eqref{eq:aux346} implies that, for $n'$ large enough, there exists $A\se \F_q^{n+n'}$ such that 
$\Lambda_L(A)+\Lambda_L(\compl{A}) < 2^{k-1}.$ Thus, $L$ is uncommon. \end{proof}

We now turn our attention to the special case of $(2\times 5)$-systems with $s(L)=4.$ 
\thref{t:uncommon} (and the characterisation of one-equation systems) implies that if {\em all} the shortest equations induced by $L$  are not Sidorenko, then $L$ is not Sidorenko. 
\thref{cor:sidfun}(ii) allows us to extend this to the case when only four of the five shortest equations are not Sidorenko, under some additional constraints. 
To describe those constraints we need the following. For $a_1,\ldots,a_4, b \in \F_q$, some $n\ge 1$, the linear form $\sum_{i \in [4]} a_i x_i$ is called \emph{$b$-coincidental} if for every $i$ there is $j$ such that $a_i a_j^{-1} \in\{\pm b,\pm b^{-1}\}$. 

\begin{thm} \thlabel{p:sid-one}
Let $q$ be an odd prime power, and let $L$ be a $(2 \times 5)$-system over $\F_q$ with $s(L)=4$. Suppose that $L_1$ is Sidorenko with coefficients $\pm 1, \pm b \in \F_q$, where $b \neq \pm 1$, and that $L_2, \dots L_5$ are not Sidorenko. If $L_2$ is not $b$-coincidental, then $L$ is not Sidorenko.
\end{thm}

We remark that in fact we prove that such an $L$ is either not translation-invariant or uncommon. Furthermore, we believe that the requirement of one of the four non-Sidorenko equations to be not $b$-coincidental is a very weak one as it seems difficult to construct a $b$-coincidental  equation $L_2$ such that all three linear combinations $L_3, L_4, L_5$ are also $b$-coincidental. 

\begin{proof}[Proof of \thref{p:sid-one}] 
    We may assume that $L$ is translation-invariant, as otherwise it is not Sidorenko. 
    For $i \in \{ 2, 3, 4, 5\}$, denote the set of coefficients of $L_i$ by $\{a_{i1}, a_{i2}, a_{i3}, a_{i4} \}$ in no particular order, with possible repetitions. Since $L_2$ is not $b$-coincidental, we may assume that $a_{21} = 1$ and that $a_{22}, a_{23}, a_{24}$ do not take values in $\{\pm b, \pm b^{-1} \}$. 

 Towards using \thref{cor:sidfun} we define a function $f: \F_q \to [0, 1]$ via its Fourier coefficients  to ensure that $\sum_{i=1}^5 \tau_{L_i}(f) $ is negative. Once  $\widehat f(h) $ is fixed, we automatically set $\widehat f(-h) = {\widehat f(h)}$ (noting that there is no need for complex conjugation since our Fourier coefficients will take real values). Firstly, set $\widehat f (0) := 1/2$ so that $\E f=1/2$, and $\widehat f (1) := 1/8$. 
Moreover, let $\eps = 2^{-10}$. If $j \in \{ 2, 3, 4\}$ and  $a_{2j} \neq \pm 1$, then set $\widehat f (a_{2j}) := \pm \eps,$ where the sign is chosen uniformly at random, choices for distinct $a_{2j}$ being independent (note that when $a_{2j}= \pm 1$, $\widehat f(a_{2j})$ is already set). For all other $h$, set $\widehat f (h) := 0$. In particular, there are only eight  values of $h \neq 0$ with $\widehat f(h) \neq 0$, and 
	\begin{equation}		\label{eq:bhat}
		\widehat f(b) = \widehat f (b^{-1}) = 0
	\end{equation} by assumption on $a_{22}, a_{23}, a_{24}$ and since $b\neq \pm 1.$
We note that $f$ is indeed a function taking values in $[0,1]$ using the inverse transform 
$f(x) = \sum_{r \in \widehat{\Fqn}} \widehat{f}(r)\exp(2\pi i/p)^{- \Tr(r\cdot x)}.$ 
 For $i \in \{2, 3, 4, 5 \}$ and $h \in \widehat{\F_q} \setminus \{0 \}$  set 
$X_i(h) :=\widehat f(a_{i1}h) \widehat f(a_{i2} h) \widehat f(a_{i3} h) \widehat f(a_{i4} h),$ and note that 
\begin{align}\label{eq:sum-taus}
&    \sum_{i=2}^5 \tau_{L_i}(f) = \sum_{i=2}^5 \sum_{h\in\F_q^{\times}}X_i(h). 
\end{align} 
Our aim is to show that there is a choice of signs in the definition of the random Fourier coefficients of $f$ such that the dominant terms in~\eqref{eq:sum-taus} are negative. We need the following. 
 \begin{claim} \label{c:signs-uniform}
 For all $i \in \{2, 3, 4, 5 \}$ and $h \in \widehat{\F_q} \setminus \{0 \},$ if $X_i(h) \neq 0$, then $X_i(h)$ is positive with  probability $1/2$ and negative with probability $1/2$. 
\end{claim}
	\begin{proof}
    We claim that there is $a_{i\ell} \in \{a_{i1}, a_{i2}, a_{i3}, a_{i4} \}$, such that for all $j \in [4]\setminus \{\ell\}$, we have $a_{ij} \not\in \{a_{i\ell},-a_{i\ell}\}$. 
    That is, $ a_{i\ell}$ appears only once in the multiset $\{a_{i1},  a_{i2}, a_{i3}, a_{i4}\}$ and $-a_{i\ell}$ is not an element of it. 
	This suffices to prove the claim, as the sign of $\widehat f(a_{i\ell} h)$ is independent of the remaining three factors.

	Assume the opposite, so without loss of generality, $a_{i1}, a_{i2}\in \{\pm \alpha\}$ and $a_{i3}, a_{i4}\in\{\pm \beta\}$ for some $\alpha, \beta \in \F_q^{\times}$. By assumption, $L_i$ is not Sidorenko, so (by the one-equation characterisation) its coefficients cannot be partitioned into pairs, each summing to zero. Furthermore, we assumed $L$, and thus $L_i,$ is translation-invariant, i.e.~$a_{i1}+a_{i2}+a_{i3}+a_{i4}=0.$ Thus, we reach one of the following conclusions
		\begin{align*}
		&\alpha+\alpha +  \beta +  \beta =0, \text{ or }\ \alpha+\alpha +  \beta - \beta =0, \text{ or }\ \alpha - \alpha + \beta +  \beta =0 .
		\end{align*}
	This implies that $\alpha+\beta=0$, or $\alpha=0$, or $\beta=0,$ respectively, since $q$ is odd. This contradicts our assumptions, and the claim follows. \end{proof}
        
		Let $\xi $ be the maximum modulus of $X_i(h)$ over all $i \in \{ 2,3,4,5\}$ and $h\neq 0$. Note that $\xi \geq \eps^3/8 $ since $|X_2(1) |=|\widehat f( 1) \widehat f(a_{22} ) \widehat f(a_{23} ) \widehat f(a_{24} )| = \eps^3/8$ by construction.
        Moreover, let $\zeta(f)$ be the sum of all those terms $X_i(h)$ in~\eqref{eq:sum-taus} of modulus $\xi$. We claim that there is a choice of signs for $\widehat f (a_i)$ such that
	\begin{equation} \label{eq:zeta} 
		\zeta(f) \leq -\xi \leq  -\frac 18 \eps^3.
	\end{equation}
    The proof follows the argument from~\cite{fpz19} and~\cite{klm21-uncommon}. Namely, the expectation of $\zeta(f)$ (taken over the random sign choices) is zero since $\E[X_i(h)] = 0,$ by Claim~\ref{c:signs-uniform}. Moreover, $\zeta(f)>0$ is attained when $\widehat f(h) \geq 0$ for all $h$, which occurs with positive probability. Therefore, with positive probability, $\zeta(f)<0$,  and from here onwards, we fix $f$ so that $\zeta(f)<0$. But $\zeta(f)$ is by definition a sum of terms of modulus $\xi$, so indeed, $\zeta(f) \leq -\xi$, as required. 
    
  	Next, we claim that the sum of all terms $X_i(h)$ in~\eqref{eq:sum-taus} not in $\zeta(f)$ is at most $  2^8 \eps \xi$. This holds since all terms not in $\zeta(f)$ have modulus at most $8 \eps \xi$, and there are at most $4 \times 8 = 32$ non-zero terms $X_i(h)$. 

    Finally, we claim  that $\tau_{L_1}(f)$ is dominated by $\zeta(f)$ as follows. 
    Note that every summand in $\tau_{L_1}(f)$ is of the form $\widehat f(h)\widehat f(-h)\widehat f(bh)\widehat f(-bh) = \widehat f(h)^2 \widehat f (bh)^2$ which is either zero or $\eps^4.$ Indeed,   
the only terms possibly involving Fourier coefficients of modulus $1/8$ are $\widehat f(1)^2 \widehat f (b)^2 = 0$ and $\widehat f(b^{-1})^2 \widehat f (1)^2 =0$ using~\eqref{eq:bhat}. Moreover,  $\widehat f(h)^2 \widehat f (bh)^2 = \eps^4$  for at most eight values of $h\neq 0$. 

Therefore
	\begin{equation}
	\sum_{i=1}^5 \tau_{L_i}(f)
			\leq 8\eps^4 - \xi + 2^8\eps \xi 
		\leq (2^6\eps -1 + 2^8 \eps) \xi 
		<-\xi/2,
		\end{equation}
	recalling that 
    $\eps = 2^{-10}$, 
 	which implies the theorem by \thref{cor:sidfun}(ii).
\end{proof}

\thref{p:sid-one} indicates that the presence of just one Sidorenko equation among the five shortest equations $L_i$ is likely not to be enough for the whole system to be Sidorenko. What if two or more of the shortest equations are Sidorenko? We do not have a satisfying answer for these situations yet. Instead, we present two examples of systems where exactly three of the five shortest equations are Sidorenko. Both of these are common, neither proof of which gives that they are also Sidorenko. In fact, we can verify that the first example is not Sidorenko. We discuss both examples further in Section~\ref{sec:conclusion}. 

\begin{example}	\thlabel{ex:non-AQ}
Consider the system $L$ generated by $(x_1 - x_3) + 2(x_4 - x_5)$ and $(x_2 - x_4) + 2(x_3 - x_5).$ 
Note that the linear forms $L_1,\ldots, L_5$ have coefficient matrix
			$$\begin{pmatrix}
				1 & 0 & -1 & 2& -2\\
				0 & 1 & 2 & -1 & -2 \\
				2 & 1 & 0  &  3 & -6 \\
				1 & 2 & 3 &  0 & -6 \\
				1 &-1 & -3 & 3 & 0 \\
			\end{pmatrix},$$
and that $L_1$, $L_2$ and $L_5$ are Sidorenko. 
We claim that $L$ is common over $\F_q$, for any prime $q>3$, where the coefficients are taken mod $q$. 		
Indeed, let $n\ge 1$ and let $f:\Fqn\to[0,1]$ be a function. 
Consider $$\sum_{i=1}^5\tau_{L_i}(f) = 
    \sum_{h \in \widehat{\Fqn} \setminus \{0\}}
        2|\widehat{f}(h)|^2|\widehat{f}(2h)|^2 + |\widehat{f}(h)|^2|\widehat{f}(3h)|^2 + 2\widehat{f}(h)\widehat{f}(2h)\widehat{f}(3h)\widehat{f}(-6h).$$
Now, using the Cauchy--Schwarz Inequality, we obtain 
\begin{align*}
\sum_h \widehat{f}(h)\widehat{f}(2h)\overline{\overline{\widehat{f}(3h)\widehat{f}(-6h)}} 
    &\le \left(\sum_h |\widehat{f}(h)|^2|\widehat{f}(2h)|^2\right)^{1/2} \left(\sum_h |\widehat{f}(3h)|^2|\widehat{f}(-6h)|^2\right)^{1/2} \\
    &= \sum_h |\widehat{f}(h)|^2|\widehat{f}(2h)|^2,
\end{align*}
where the last equality follows using a substitution $h' = 3h$ and the fact that $|\widehat f (6h)| = |\widehat f (-6h)|$.
Thus, by \thref{cor:sidfun}(i), the system is common. 

We also found that over $\F_5$, the set $A\se \F_5^2$ consisting of the elements 
$$\dbinom{0}{0},\dbinom{0}{3}, \binom{1}{2},\dbinom{3}{0}, \dbinom{3}{3}, \dbinom{4}{0}, \dbinom{4}{1}, \dbinom{4}{2} $$
has fewer than $8^5/5^4$ solutions to $L=0$ and thus, $L$ is not  Sidorenko over $\F_5$. 
\end{example}	      

Our second example illustrates the power of containing an additive quadruple (denoted AQ), which can be used to `dominate' other critical equations, at least for showing that a system is common. Note that we explicitly exclude AQs in \thref{p:sid-one}. We do not know whether this is necessary for \thref{p:sid-one} to be true. 

\begin{example}\thlabel{ex:cont-AQ}
Consider the system generated by $x_1 - x_2 + x_3 - x_4$ and $(x_1 - x_3) + 2(x_2 - x_5).$ The linear forms $L_1,\ldots L_5$ have coefficient matrix  
$$\begin{pmatrix}
1 & -1 & 1 & -1 & 0\\
1 & 2 & -1 & 0 & -2\\
0 & 3 & -2 & 1 & -2\\
3 & 0 & 1 & -2 & -2\\
2 & 1 & 0 & -1 & -2.
\end{pmatrix}.$$
Again, exactly three out of five of the shortest induced equations are Sidorenko. 

We claim that $L$ is common over $\F_q$ for all prime powers $q$ except powers of $2$ and $3$. The proof of this claim is significantly different from the proof in the previous example. Considering the shortest equations we obtain for any $f:\Fqn\to[0,1]$ that 
\begin{align*}
 \sum_{i=1}^5\tau_{L_i}(f) 
 &=\sum_{h \in \widehat{\Fqn} \setminus \{0\}}
   |\widehat{f}(h)|^4 
    + 2 |\widehat{f}(h)|^2|\widehat{f}(2h)|^2 +  2\widehat{f}(-h)\widehat{f}(2h)^2\widehat{f}(-3h)\\
 &= \sum_{h \in \widehat{\Fqn} \setminus \{0\}} 
    \frac 12 |\widehat f(2h)|^4  + \frac 12 |\widehat f(3h)|^4  
    + 2 |\widehat f(h)|^2 |\widehat f(-2h)|^2 
    + 2\, \mathrm{Re} \big( \widehat f(h) \widehat f(3h) \widehat f (-2h)^2 \big),
\end{align*}
by suitable index shifts in some summands, $\widehat{f}(-2h) =\overline{\widehat{f}(2h)},$ and since $\tau_{L_i}(f)$ is real for a real-valued function $f$ by~\eqref{eq:twisted} and definition of $\Lambda_{L_i}(f).$ 
Now, 
\begin{align*}
&\frac 12 |\widehat f(2h)|^4  + \frac 12 |\widehat f(3h)|^4  
    + |\widehat f(h)|^2 |\widehat f(-2h)|^2 
    + 2\, \mathrm{Re} \big( \widehat f(h) \widehat f(3h) \widehat f (-2h)^2 \big)\\
&\qquad \geq |\widehat f (-2h) \widehat f(3h)|^2 + |\widehat f(h)|^2 |\widehat f(-2h)|^2 + 2\, \mathrm{Re} \big( \widehat f(h) \widehat f(3h) \widehat f (-2h)^2 \big)\\
&\qquad = \left(\widehat f (h) \widehat f(-2h) + \overline{\widehat f(-2h) \widehat f (3h)} \right) \left( \overline {\widehat f (h) \widehat f(-2h)} + \widehat f(-2h) \widehat f (3h) \right) \geq 0
\end{align*}
for all $h$. Thus, the system is common by \thref{cor:sidfun}(i). \end{example}

\section{Discussion and further research directions}\label{sec:conclusion}

In this paper, we prove that a system $L$ is not Sidorenko whenever $s(L)$ is odd. We also make significant progress towards understanding the case when $s(L)$ is even, though much remains to discover. In the case when $L$ is a single equation with an even number of non-zero coefficients, the work in~\cite{fpz19} implies that $L$ is Sidorenko if and only if it is common. However, for systems of multiple equations, such a correspondence cannot be expected as \thref{ex:non-AQ} demonstrates that there are systems with $s(L)$ even that are common but not Sidorenko. The counter-example that shows it is not Sidorenko was found computationally and currently we do not have a good approach that can be used to show that similar systems are not Sidorenko. 

In Section~\ref{sec:2xk}, we make progress towards a full characterisation of Sidorenko $(2 \times k)$-systems (where the first interesting case is $s(L)=4$). In light of our work there, we make the following conjecture.

\begin{conj}\thlabel{c:25}
    All $(2 \times 5)$-systems with $s(L) =4$ are not Sidorenko.
\end{conj}
This is true when all the shortest equations $L_i$ are not Sidorenko, by \thref{t:uncommon}. \thref{p:sid-one} extends this to the case when only one of the shortest equations is Sidorenko (under some additional constraints). Apart from those cases and  \thref{ex:non-AQ}, all other cases remain open.

In \thref{p:sid-one}, we exclude the case when the Sidorenko equation is an additive quadruple, i.e.~ $x+y=z+w.$ The presence of an additive quadruple among the shortest equations seems to be a bottleneck towards proving \thref{c:25}. \thref{ex:cont-AQ} contains an additive quadruple, so proving that this system is not Sidorenko would be a first step towards \thref{c:25}. In this example, we crucially used the additive quadruple to `dominate' the other shortest equations to show commonness. We wonder whether this is an exception, and thus \thref{ex:cont-AQ} is a special case.
\begin{ques}
    For $k\ge 5,$ does there exist an uncommon $(2\times k)$-system containing an additive quadruple?
\end{ques}
We remark that for $k = 4,$ this is indeed the case since we proved in~\cite{klm21-uncommon} that any irredundant $(2\times 4)$-system is uncommon. However, in such systems, $s(L)\le 3$ and so the additive quadruple is not really a `building block' in this case.

If \thref{c:25} is true, then we believe it is possible that all irredundant $(2 \times k)$-systems with $s(L)= k-1$ are non-Sidorenko. 
\begin{ques}\thlabel{qu:2k}
    Does there exist a $(2 \times k)$-system with $s(L)=k-1$ that is Sidorenko, when $k \ge 7$ is odd?
\end{ques}

Let us give yet another motivation behind \thref{c:25,qu:2k}. In Section~\ref{sec:entropy}, we show that the union of certain Sidorenko systems with particular shared variables is also Sidorenko. There, we combined several Sidorenko equations on `blocks'. 
When considering the union of Sidorenko systems with some shared variables, it seems that the more the variables `intermingle', the harder it is for the combined system to be Sidorenko. For example, the system $y-x = z-w$ and $3y+z= 3w+x$ is a system of two Sidorenko equations with shared variables on both sides. This is a four term arithmetic progression, which is an irredundant $(2\times 4)$-system and hence uncommon (see~\cite{klm21-uncommon}). \thref{ex:non-AQ} also provides a non-Sidorenko 5-variable system where the coefficients of two Sidorenko equations intermingle. 

We wonder whether Sidorenko equations can be combined in arbitrary ways (obeying the $s(L)$ being even condition) and yield Sidorenko systems, but believe that some `block-type' structure of the coefficients, along the lines of \thref{thm:main2-entropy} is necessary. 
The confirmation of \thref{c:25} would support the latter, as all reasonable two-equation systems admitting a graph template (in this case an $S_2$-template) have at least six variables. 

Another piece of evidence supporting the hypothesis that all Sidorenko systems are structured in a `block-like' fashion, is that all the systems we know to be Sidorenko are formed by combining sets of Sidorenko systems. We do not think it is possible to create a Sidorenko system in another way. 
\begin{ques}
    Does there exist a Sidorenko system whose matrix is not equivalent to a matrix where every row is a Sidorenko equation?
\end{ques}

Given the large family of \emph{graphs} that are proven to be Sidorenko via use of the entropy method, one may wonder whether \thref{thm:main2-entropy} could be strengthened to include systems whose graph templates contain cycles. Many such systems with \emph{repeated} linear forms on distinct tuples are equivalent to systems that are proven to be Sidorenko by \thref{thm:main2-entropy}. For example, the system $L_1(\xv^{(1)})=L_1(\xv^{(2)}),$ $L_1(\xv^{(1)})=L_1(\xv^{(3)}),$  $L_2(\xv^{(2)})=L_2(\xv^{(4)}),$ $L_2(\xv^{(3)})=L_2(\xv^{(4)})$ has a graph template that is a $C_4$, and the symmetry of repeated forms makes it amenable to the entropy method. However, this system is also equivalent to one with a graph template that is a star with one double edge, so is covered by \thref{thm:main2-entropy}. With this in mind, we wonder about the following (which would also form another piece of the puzzle towards determining the structure of Sidorenko systems).

\begin{ques}
    Does there exist a Sidorenko system $L$ admitting a graph template which is a $C_4$, where $L$ is not equivalent to any system admitting a graph template that is a tree?
\end{ques}

Finally, we want to comment that both \thref{ex:non-AQ,ex:cont-AQ} are common, but the proofs use very different underlying properties of the systems. The argument for \thref{ex:non-AQ} relies on the \emph{multiplicative} structure of the coefficients. 
The multiplicative structure of the coefficients also plays a crucial role in \thref{p:sid-one}. As noted in the introduction, this contrasts the situation for single-equation systems where the characterisation of Sidorenko and common equations merely depends on the additive structure of the coefficients. The additive condition of being translation invariant is a necessary condition for a system to be Sidorenko. For commonness, translation-invariance is not necessary since equations with an odd number of variables are common. However, we do not know whether it is necessary for a system of two or more equations to be common.
    \begin{ques}
        Does there exist a system of rank at least two which is common, but not translation-invariant?
    \end{ques}

\bigskip

\noindent    
The first and third authors would like to thank UNSW for their hospitality and for providing a stimulating research environment, which is where this research began.  

   \bibliographystyle{abbrv} 

\end{document}